\documentclass[a4paper,reqno]{amsart}

\usepackage{amsmath}
\usepackage{mathptmx}

\usepackage{amssymb, amsthm,  amsfonts}
\usepackage{amsmath}
\usepackage{latexsym}
\usepackage{euscript}

\usepackage{mathrsfs}
\usepackage{setspace}
\usepackage{paralist}

\usepackage{mathrsfs}

\usepackage{float}
\usepackage{pgfplots}
\pgfplotsset{compat=1.12}
\usepackage{pgf,tikz}
\usetikzlibrary{arrows}
\tikzset{every mark/.append style={scale=0.1}}

      \def\dC{{\mathbb C}}

      \def\dR{{\mathbb R}}

      \def\cF{{\mathcal F}}
   \def\cH{{\mathcal H}}   
   \def\cK{{\mathcal K}}   
\def\cM{{\mathcal M}}   \def\cN{{\mathcal N}}

\let\xker=\ker \def\ker{{\xker\,}}
\newcommand{\Pluspunkt}{\stackrel{.}{+}}

\DeclareMathOperator{\RE}{Re}
\DeclareMathOperator{\IM}{Im}

\DeclareMathOperator{\sgn}{sgn}

\newcommand{\sigmap}{\sigma_{\rm p}}

\newcommand{\rd}{\mathrm{d}}

\newcommand{\Skdef}{\langle\cdot,\cdot\rangle}
\newcommand{\Skindef}{[\cdot,\cdot]}

\newcommand{\matriz}[4]{\left[\;\begin{matrix} #1 & #2 \\[0.2ex]
#3 & #4 \end{matrix}\;\right]}
\newcommand{\vect}[2]{\left[\;\begin{matrix} #1 \\[0.2ex] #2 \end{matrix}\;\right]}

\newcommand{\ra}{\rightarrow}
\newcommand{\ol}{\overline}
\newcommand{\PI}[2]{\left\langle #1, #2\right\rangle}

\newcommand{\K}[2]{\left[ #1, #2\right]}
\newcommand{\ort}{[\perp]}
\newcommand{\sdo}{[\dotplus]}

\DeclareMathOperator{\linspan}{span}

\newcommand{\defeq}{\mathrel{\mathop:}=}

\newtheorem{theorem}{Theorem}[section]
\newtheorem{proposition}[theorem]{Proposition}
\newtheorem{corollary}[theorem]{Corollary}

\theoremstyle{definition}
\newtheorem{definition}[theorem]{Definition}

\newtheorem{remark}[theorem]{Remark}

\newcounter{counter_a}
\newenvironment{myenum}{\begin{list}{{\rm(\roman{counter_a})}}%
{\usecounter{counter_a}
\setlength{\itemsep}{0.5ex}\setlength{\topsep}{0.7ex}
\setlength{\leftmargin}{5ex}\setlength{\labelwidth}{5ex}}}{\end{list}}

\newcounter{counter_b}

\newcommand\void[1]{}

\numberwithin{equation}{section}

\definecolor{darkgreen}{rgb}{0.,0.6,0.}

{\left[\begin{smallmatrix}}
{\end{smallmatrix}\right]}%

\newcommand\proj[2]{P_{#1/\!/#2}}

\begin{document}
\title{Spectrum of $J$-frame operators}

\author[J.~Giribet]{Juan Giribet}
\address{Departamento de Matem\'atica -- FI-UBA and \newline
Instituto Argentino de Matem\'{a}tica ``Alberto P.~Calder\'{o}n'' (CONICET), \newline
Saavedra 15 (1083) Buenos Aires, Argentina}
\email{jgiribet@fi.uba.ar}

\author[M.~Langer]{Matthias Langer}
\address{Department of Mathematics and Statistics,
University of Strathclyde, \newline
26 Richmond Street,
Glasgow G1 1XH, United Kingdom}
\email{m.langer@strath.ac.uk}

\author[L.~Leben]{Leslie Leben}
\address{Institut f\"ur  Mathematik,
Technische Universit\"{a}t Ilmenau, \newline
Postfach 100565, D-98684 Ilmenau,
Germany}
\email{leslie.leben@tu-ilmenau.de}

\author[A.~Maestripieri]{Alejandra Maestripieri}
\address{Departamento de Matem\'atica -- FI-UBA and \newline
Instituto Argentino de Matem\'{a}tica ``Alberto P.~Calder\'{o}n'' (CONICET), \newline
Saavedra 15 (1083) Buenos Aires, Argentina}
\email{amaestri@fi.uba.ar}

\author[F.~Mart\'{\i}nez Per\'{\i}a]{Francisco Mart\'{\i}nez Per\'{\i}a}
\address{Departamento de Matem\'{a}tica -- FCE-UNLP and \newline
Instituto Argentino de Matem\'{a}tica ``Alberto P.~Calder\'{o}n'' (CONICET), \newline
Saavedra 15 (1083) Buenos Aires, Argentina}
\email{francisco@mate.unlp.edu.ar}

\author[C.~Trunk]{Carsten Trunk}
\address{Institut f\"ur  Mathematik,
Technische Universit\"{a}t Ilmenau, \newline
Postfach 100565, D-98684 Ilmenau,
Germany}
\email{carsten.trunk@tu-ilmenau.de}




\begin{abstract}
A $J$-frame is a frame $\cF$ for a Krein space $(\cH,\Skindef)$
which is compatible with the indefinite inner product $\Skindef$
in the sense that it induces an indefinite reconstruction formula
that resembles those produced by orthonormal bases in $\cH$.
With every $J$-frame the so-called $J$-frame operator is associated,
which is a self-adjoint operator in the Krein space $\cH$.
The $J$-frame operator plays an essential role in the indefinite reconstruction formula.

In this paper we characterize the class of $J$-frame operators in a Krein
space by a $2\times 2$ block operator representation. The $J$-frame
bounds of $\cF$ are then recovered as the suprema and
infima of the numerical ranges of some uniformly positive
operators  which are build from the entries of the
$2\times 2$ block representation. Moreover, this $2\times 2$ block representation
is utilized to obtain enclosures for the spectrum of $J$-frame operators,
which finally leads to the construction of a square root. This square root
allows a complete description of all $J$-frames associated
with a given $J$-frame operator.
\end{abstract}

\maketitle

\section{Introduction}

\noindent
Frame theory is a key tool in signal and image processing, data compression,
and sampling theory, among other applications in engineering,
applied mathematics, and computer sciences.
The major advantage of a frame over an orthonormal, orthogonal,
or Riesz basis is its redundancy: each vector admits several
different reconstructions in terms of the frame coefficients.
For instance, frames have shown to be useful in signal processing applications
when noisy channels are involved, because a frame allows to
reconstruct vectors (signals) even if some of the frame coefficients
are missing (or corrupted); see \cite{BP05,HP04,SH03}.

A frame for a Hilbert space $(\cH,\Skdef)$ is an, in general, redundant (overcomplete)
family of vectors $\cF=\{f_i\}_{i\in I}$ in $\cH$ for which there exists
a pair of positive constants $0<\alpha\leq \beta$ such that
\begin{equation}\label{Schwaben}
  \alpha\, \|f\|^2 \leq \sum_{i\in I} |\PI{f}{f_i}|^2 \leq \beta\, \|f\|^2
  \qquad \text{for every}\;\; f\in \cH.
\end{equation}
Note that these inequalities establish an equivalence between the
norm of the vector $f\in \cH$ and the $\ell_2$-norm of its
frame coefficients $\{\PI{f}{f_i}\}_{i\in I}$.

Every frame $\cF$ for $\cH$ has an associated frame operator $S:\cH\ra\cH$,
which is uniformly positive (i.e.\ positive-definite and boundedly invertible)
in $\cH$, and it allows to reconstruct each $f\in \cH$ as follows:
\begin{equation*}
  f = \sum_{i\in I}\PI{f}{f_i}\, S^{-1}f_i=\sum_{i\in I}\PI{f}{S^{-1}f_i}f_i.
\end{equation*}
Given a uniformly positive operator $S:\cH\ra\cH$ it is possible to describe
the complete family of frames whose frame operator is $S$.
Indeed, if $\cF=\{f_i\}_{i\in I}$ is a frame for $\cH$ then the upper
bound in \eqref{Schwaben} implies that the so-called synthesis operator $T:\ell_2(I)\ra\cH$,
\[
  Tx = \sum_{i\in I}\PI{x}{e_i}f_i,
\]
is bounded,
where $\{e_i\}_{i\in I}$ is the standard orthonormal basis in $\ell_2(I)$.
Then its frame operator is $S$ if and only if
\[
  S=TT^*.
\]
Therefore, $S$ is the frame operator of $\cF$ if and only if the
(left) polar decomposition of $T$ can be written as
\[
  T = S^{1/2}V,
\]
for some co-isometry $V:\ell_2(I)\ra\cH$; see for instance \cite[Corollary~2.15]{CPS}.
Hence, the family of frames whose frame operator is $S$ is in a
one-to-one correspondence with the family of co-isometries onto the Hilbert space $\cH$.
Also, note that this description is possible due to the existence
of a (positive) square root of $S$.

\medskip

Recently, various approaches to introduce frame theory on Krein spaces have
been suggested; see \cite{EFW,GMMM12,PW}.  In the current paper we further
investigate the notion of $J$-frames that was proposed in \cite{GMMM12}.
This concept was motivated by a signal processing problem,
where signals are disturbed with the same energy at high and low band frequencies;
see the discussion at the beginning of Section~3 in \cite{GMMM12}.

Given a Bessel family (that is a family which satisfies
the upper bound in \eqref{Schwaben}) of vectors $\cF=\{f_i\}_{i\in I}$
in a Krein space $(\cH,\Skindef)$, we divide it in a family $\cF_+=\{f_i\}_{i\in I_+}$
of non-negative vectors in $\cH$ and a family $\cF_-=\{f_i\}_{i\in I_-}$
of negative vectors in $\cH$.
Roughly speaking, $\cF=\{f_i\}_{i\in I}$ is a $J$-frame for $(\cH,\Skindef)$
if the ranges of the synthesis operators $T_+$ and $T_-$ associated
with $\cF_+$ and $\cF_-$ are a maximal uniformly positive subspace and
a maximal uniformly negative subspace of $\cH$, respectively
(see Subsection~\ref{Jframes} for further details).
In particular, it is immediate that orthonormal bases in Krein spaces
are $J$-frames as they generate maximal dual pairs \cite[Section~1.10]{AI89}.

Each $J$-frame for $(\cH,\Skindef)$ is also a frame
(in the Hilbert space sense) for $\cH$.  Moreover, $\cF_+=\{f_i\}_{i\in I_+}$
and $\cF_-=\{f_i\}_{i\in I_-}$ are frames for the Hilbert spaces $(R(T_+),\Skindef)$
and $(R(T_-),-\Skindef)$, respectively.  However, it is easy to construct
frames for $\cH$ which are not $J$-frames; see \cite[Example~3.3]{GMMM12}.

For each $J$-frame for $\cH$ one can define a $J$-frame operator $S:\cH\ra\cH$
(see \eqref{Iphofen} below), which is a bounded and boundedly invertible
self-adjoint operator in $(\cH,\Skindef)$.
It can be used to describe every vector in $\cH$ in terms of the
frame vectors $\cF=\{f_i\}_{i\in I}$:
\[
  f = \sum_{i\in I} \sigma_i\K{f}{S^{-1}f_i}f_i= \sum_{i\in I} \sigma_i\K{f}{f_i}S^{-1}f_i,
  \qquad f\in\cH,
\]
where $\sigma_i=\sgn\K{f_i}{f_i}$.
This is known as the indefinite reconstruction formula associated with $\cF$
since it resembles the reconstruction formula provided by an orthonormal basis
in a Krein space.

If $T:\ell_2(I)\ra \cH$ is the synthesis operator of $\cF$, then $S$
is given by
\begin{equation}\label{Iphofen}
  S= TT^+,
\end{equation}
where $T^+$ stands for the adjoint of $T$
with respect to $\Skindef$.  Hence, given a $J$-frame operator $S$
in $(\cH,\Skindef)$, it is natural to look for descriptions of the
family of $J$-frames whose $J$-frame operator is $S$.
This is a non-trivial problem due to several reasons.

First, the existence of a square root of a self-adjoint operator
in a Krein space depends on the location of its spectrum.
Some characterizations of $J$-frame operators can be found
in \cite[Subsection~5.2]{GMMM12}, but none of them guarantees
the existence of a square root.

Second, in general there is no polar decomposition for linear operators
acting between Krein spaces; see \cite{BMRRR97, MRR99, MRR02, MRR05}.

\medskip

The aim of this work is to obtain a deeper insight on $J$-frame operators, in particular,
to obtain enclosures for the spectrum of a $J$-frame operator and
 a full description of the $J$-frames associated with a
prescribed $J$-frame operator. The two main results of the
paper are the following.
\begin{itemize}
  \item[\rm (i)] The spectrum of a $J$-frame operator is always contained
  in the open right half-plane with a positive distance to the imaginary
  axis. This distance can be estimated in terms of the $J$-frame bounds.
  \item[\rm (ii)] Due to the location of the spectrum according to item (i),
  there exists a square root of a $J$-frame operator. This enables one to
  show that there is a bijection between all co-isometries and all $J$-frames
  with the same $J$-frame operator.
\end{itemize}

\medskip

The paper is organized as follows.  Section~\ref{sec:pre} contains preliminaries
both on frames for Hilbert spaces and for Krein spaces.
There we recall the notions of $J$-frames and $J$-frame operators,
and we present some known results.

In Section~\ref{sec:matrix}, given a bounded self-adjoint operator $S$
acting in a Krein space $(\cH,\Skindef)$, we describe $J$-frame operators as
block operator matrices (with respect to a suitable fundamental decomposition)
such that their entries have some particular properties.
For instance, we prove that $S$ is a $J$-frame operator if and only if
there exists a fundamental decomposition $\cH=\cH_+\sdo \cH_-$ such that
\[
  S = \matriz{A}{-AK}{K^*A}{D},
\]
where $A$ is uniformly positive in the Hilbert space $(\cH_+,\Skindef)$, $K:\cH_-\ra\cH_+$
satisfies $\|K\|<1$, and $D$ is a self-adjoint operator such that $D+K^*AK$
is uniformly positive in the Hilbert space
$(\cH_-,-\Skindef)$; see Theorem~\ref{Irland} below.
Also, a dual representation is given in Theorem~\ref{Scotland} below,
where the roles of the operators in the diagonal of the
block operator matrix are interchanged.
Using these two representations we describe the inverse of $S$ and
the positive operators $S_\pm$, which are defined in \eqref{eses} below and
satisfy $S=S_+-S_-$.

Given a $J$-frame $\cF=\{f_i\}_{i\in I}$ for a Krein space $(\cH,\Skindef)$,
the $J$-frame bounds of $\cF$ are the two pairs of frame bounds associated
with $\cF_+$ and $\cF_-$ as frames for the Hilbert spaces $(R(T_+), \Skindef)$
and $(R(T_-), -\Skindef)$, respectively.
In Section~\ref{sec:framebounds} we recover the $J$-frame bounds of $\cF$ as
spectral bounds for the $J$-frame operator $S$ associated with $\cF$.
More precisely, we describe the $J$-frame bounds of $\cF$ in terms of
suprema and infima of the numerical ranges of the uniformly positive entries
in the block operator matrix representations discussed above.

In Section~\ref{sec:spectrum} we use the representations obtained
in Section~\ref{sec:matrix} to obtain enclosures for the real part and
the non-real part of the spectrum of a $J$-frame operator $S$.
In particular, we show that the spectrum of a $J$-frame operator is
always contained in the right half-plane $\{z\in\dC:\, \RE z>0\}$.
Moreover, given a $J$-frame $\cF=\{f_i\}_{i\in I}$ with $J$-frame operator $S$,
we describe the spectral enclosures of $S$ in terms of the $J$-frame bounds of $\cF$.

Finally, in Section~\ref{sec:squareroot} we use the Riesz--Dunford functional calculus
to obtain the square root of a $J$-frame operator.  We characterize the family
of all $J$-frames for $\cH$ whose $J$-frame operator is $S$.
More precisely, the synthesis operator $T:\ell_2(I)\ra\cH$ of a $J$-frame $\cF$
satisfies $TT^+=S$ if and only if
\[
  T = S^{1/2}U,
\]
where $S^{1/2}$ is the square root of $S$ and $U:\ell_2(I)\ra\cH$
is a co-isometry between Krein spaces. Hence, there is a bijection between
all co-isometries and all  $J$-frames with the same $J$-frame operator.

\section{Preliminaries}\label{sec:pre}

\noindent
When $\cH$ and $\cK$ are Hilbert spaces, we denote by $L(\cH, \cK)$ the vector space
of bounded linear operators from $\cH$ into $\cK$ and by $L(\cH)=L(\cH,\cH)$
the algebra of bounded linear operators acting on $\cH$.

For $T\in L(\cH)$ we denote by $\sigma(T)$ and $\sigmap(T)$ the \emph{spectrum}
and the \emph{point spectrum} of $T$, respectively.
Moreover, we denote by $N(T)$, $R(T)$ and $W(T)$ the \emph{kernel}, the \emph{range}
and the \emph{numerical range} of $T$, respectively.

\medskip

Given two closed subspaces $\cM$ and $\cN$ of $\cH$, hereafter $\cM\dotplus \cN$
denotes their direct sum.
Moreover, in case when $\cH=\cM \dotplus\cN$, we denote
the (unique) projection with range $\cM$ and kernel $\cN$ by $\proj{\cM}{\cN}$.
In particular, if $\cN=\cM^\bot$ then $P_\cM=\proj{\cM}{\cM^\bot}$
denotes the orthogonal projection onto $\cM$.

\subsection{Frames for Hilbert spaces}\label{sec:frames}

\noindent
We recall the standard notation for frames for Hilbert spaces and some basic results;
see, e.g.\ \cite{TaF,CK13,Chr,HanLarson}.

\medskip

A \emph{frame} for a Hilbert space $\cH$ is a family of
vectors $\cF=\{f_i\}_{i\in I}$, $f_i\in\cH$, for which there exist
constants $0 < \alpha \le \beta < \infty$ such that
\begin{equation}\label{ecu frames}
  \alpha\,\|f\|^2 \le \sum_{i\in I} |\langle f,f_i\rangle |^2 \le \beta\,\|f\|^2
  \qquad \text{for every $f\in \cH$}.
\end{equation}
The optimal constants (maximal for $\alpha$ and minimal for $\beta$) are called
the \emph{upper} and the \emph{lower frame bounds of $\cF$}, respectively.

If a family of vectors $\cF=\{f_i\}_{i\in I}$ satisfies the upper bound condition
in \eqref{ecu frames}, then $\cF$ is a \emph{Bessel family}.
For a Bessel family $\cF=\{f_i\}_{i\in I}$ the \emph{synthesis operator}
$T\in L(\ell_2(I),\cH)$ is defined by
\[
  Tx = \sum_{i\in I}\PI{x}{e_i}f_i, \qquad x\in\ell_2(I),
\]
where $\{e_i\}_{i\in I}$ is the standard orthonormal basis of $\ell_2(I)$.
A Bessel family $\cF$ is a frame for $\cH$ if and only if $T$ is surjective.
In this case, the operator $S \defeq TT^*\in L(\cH)$ is uniformly positive
and called \emph{frame operator}.  It can easily be verified that
\begin{equation}\label{ecu S}
  Sf = \sum_{i\in I}\PI{f}{f_i}f_i \qquad \text{for every $f\in\cH$}.
\end{equation}
This implies that the frame bounds of $\cF$ can be computed as
$\alpha=\|S^{-1}\|^{-1}$ and $\beta=\|S\|$.

From \eqref{ecu S} it is also easy to obtain the \emph{frame reconstruction formula}
for vectors in $\cH$:
\[
  f = \sum_{i\in I}\PI{f}{S^{-1}f_i}f_i
  = \sum_{i\in I}\PI{f}{f_i}S^{-1}f_i
  \qquad \text{for every $f\in\cH$};
\]
and the frame $\{S^{-1}f_i\}_{i\in I}$ is called the \emph{canonical dual frame} of $\cF$.

\subsection{Krein spaces}

We recall the standard notation and some basic results on Krein spaces.
For a complete exposition on the subject (and the proofs of the results below)
see the books by Azizov and Iokhvidov \cite{AI89} and Bogn\'ar \cite{B74};
see also \cite{A79,DR,K70,R02}.

A vector space $\mathcal H$ with a Hermitian sesquilinear form $\Skindef$ is
called a \emph{Krein space} if there exists a so-called
\emph{fundamental decomposition}
\[
  \mathcal H= \mathcal H_+  \sdo \mathcal H_-,
\]
which is the direct (and orthogonal with respect to $\Skindef$) sum of two
Hilbert spaces $(\mathcal H_+, \Skindef)$ and $(\mathcal H_-, -\Skindef)$.
These two Hilbert spaces induce in a natural way a Hilbert space inner product $\Skdef$
and, hence, a Hilbert space topology  on $\mathcal H$.
Observe that the indefinite metric $\Skindef$ and the
Hilbert space inner product $\Skdef$ of $\mathcal H$ are related by means of a
\emph{fundamental symmetry}, i.e.\ a unitary self-adjoint operator $J\in L(\mathcal H)$
that satisfies
\[
  \langle x, y\rangle = [Jx, y] \qquad \text{for}\;\; x,y\in \mathcal H.
\]
Although the fundamental decomposition is not unique, the norms induced
by different fundamental decompositions turn out to be equivalent;
see, e.g.\ \cite[Proposition~I.1.2]{L82}.
Therefore, the (Hilbert space) topology in $\cH$ does not depend on the
chosen fundamental decomposition.

If $\mathcal H$ and $\mathcal K$ are Krein spaces and
 $T\in L(\mathcal H, \mathcal K)$, the adjoint operator of T is the unique
 operator $T^+\in L(\cK,\cH)$ satisfying
 \[
 [Tx,y]=[x,T^+y] \qquad \text{for every $x\in\cH$, $y\in\cK$}.
 \]
An operator $T\in L(\mathcal H)$
is \emph{self-adjoint in a Krein space} $(\mathcal H, \Skindef)$  if $T = T^+$.
The spectrum of such an operator $T$ is symmetric with respect to the real axis.

\medskip

Let $(\mathcal H, \Skindef)$ be a Krein space.
A vector $x\in\cH$ is called \emph{positive} if $[x, x] > 0$,
\emph{negative} if $[x,x]<0$ and \emph{neutral} if $[x,x]=0$.
A subspace $\mathcal L$ of $\mathcal H$ is \emph{positive} if every
$x \in \mathcal L\setminus\{0\}$ is a positive vector.
A subspace $\mathcal L$ of $\mathcal H$ is \emph{uniformly positive}
if there exists $\alpha > 0$ such that $[x, x] \geq  \alpha\|x\|^2$ for
every $x \in \mathcal L$,
where $\|\cdot\|$ stands for the norm of the associated Hilbert space $(\mathcal H, \Skdef)$.
\emph{Non-negative}, \emph{neutral}, \emph{negative}, \emph{non-positive} and
\emph{uniformly negative subspaces} are defined analogously.

Given a subspace $\cM$ of $(\cH,\Skindef)$, the \emph{orthogonal companion of $\cM$}
is defined by
\[
  \cM^{\ort} \defeq \{ x\in\cH :\, \K{x}{m}=0, \; \text{for every $m\in\cM$}\}.
\]
The \emph{isotropic part} of $\cM$, $\cM^\circ \defeq \cM\cap \cM^{\ort}$,
can be a non-trivial subspace.
A subspace $\cM$ of $\cH$ is \emph{non-degenerate} if $\cM^\circ=\{0\}$.
Otherwise, it is a \emph{degenerate} subspace of $\cH$.

A subspace $\cM$ of $\cH$ is \emph{regular} if $\cM \dotplus \cM^{\ort}=\cH$.
The subspace $\cM$ is regular if and only if there exists a (unique)
self-adjoint projection $E$ onto $\cM$; see, e.g.\ \cite[Theorem~1.7.16]{AI89}.

\subsection{Frames for Krein spaces}\label{Jframes}

\noindent
Recently, there have been various approaches to introduce frame theory on Krein spaces;
see \cite{EFW,GMMM12,PW}.
In the following we recall the notion of $J$-frames introduced in \cite{GMMM12}.
Some particular classes of $J$-frames where also considered in \cite{HKP16}.

Let $(\mathcal H, \Skindef)$ be a Krein space,
let $\cF=\{f_i\}_{i\in I}$ be a Bessel family in $\cH$
with synthesis operator $T:\ell_2(I)\ra\cH$,
and set $I_+\defeq\{i\in I:\, \K{f_i}{f_i} \ge 0\}$
and $I_-\defeq\{i\in I:\, \K{f_i}{f_i}< 0\}$.
Further, consider the orthogonal decomposition of $\ell_2(I)$ induced by the
partition of $I$,
\begin{equation}\label{desc fund}
  \ell_2(I) = \ell_2(I_+) \oplus \ell_2(I_-),
\end{equation}
denote by $P_\pm$ the orthogonal projections onto $\ell_2(I_\pm)$, respectively,
and set $T_\pm \defeq TP_\pm$, i.e.\
\begin{equation}\label{tes}
  T_\pm x = \sum_{i\in I_\pm}\PI{x}{e_i}f_i, \qquad x\in\ell_2(I).
\end{equation}

The following sets play an important role in the following:
\begin{equation}\label{emes}
  \cM_\pm\defeq \ol{\linspan\{f_i:\, i\in I_\pm\}}.
\end{equation}
They are related to the ranges of $T_+$ and $T_-$ as follows:
\[
  \linspan\{f_i:\ i\in I_\pm\}\subseteq R(T_\pm)\subseteq \cM_\pm.
\]
The ranges of $T_+$ and $T_-$ also satisfy $R(T)=R(T_+) + R(T_-)$
and play an essential role in the definition of $J$-frames.

\begin{definition}
Let $\cF=\{f_i\}_{i\in I}$ be a Bessel family in a Krein space $\cH$ and
let $T_\pm$ be as in \eqref{tes}.
Then $\cF$ is called a \emph{$J$-frame} for $\cH$
if $R(T_+)$ is a maximal uniformly positive subspace
and $R(T_-)$ is a maximal uniformly negative subspace of $\cH$.
\end{definition}

If $\cF$ is a $J$-frame for $\cH$, then, $R(T_\pm)=\cM_\pm$ and,
\begin{equation}\label{suma}
  R(T) = R(T_+) \dotplus R(T_-) = \cM_+ \dotplus \cM_-=\cH,
\end{equation}
where the last equality follows from \cite[Corollary~1.5.2]{AI89}.
Thus, $\cF$ is also a frame for the Hilbert space $(\cH, \Skdef)$.
Moreover, it is easy to see that $\cF_\pm\defeq \{f_i\}_{i\in I_\pm}$
is a frame for the Hilbert space $(\cM_\pm, \pm\Skindef)$.

\medskip
The following is a characterization of $J$-frames in terms of frame inequalities;
see \cite[Theorem~3.9]{GMMM12}.

\begin{theorem}\label{thm J frame bounds}
Let $\cF=\{f_i\}_{i\in I}$ be a frame for $\cH$.  Then $\cF$ is a $J$-frame
if and only if $\cM_\pm$ (defined as in \eqref{emes}) are non-degenerate subspaces
of $\cH$ and there exist constants $0<\alpha_\pm \leq \beta_\pm$ such that
\begin{equation}\label{eq J frame bounds}
  \alpha_\pm (\pm\K{f}{f}) \le \sum_{i\in I_\pm} \big|\K{f}{f_i}\big|^2
  \le \beta_\pm (\pm\K{f}{f})\,
  \qquad \text{for every $f\in \cM_\pm$}.
\end{equation}
\end{theorem}

When $\cF=\{f_i\}_{i\in I}$ is a $J$-frame for $\cH$, we
endow the coefficient space $\ell_2(I)$ with the following indefinite inner product:
\begin{equation*}\label{l2}
  \K{x}{y}_2 \defeq \sum_{i\in I_+} x_i \ol{y_i}- \sum_{i\in I_-} x_i \ol{y_i},
  \qquad x=(x_i)_{i\in I},\, y=(y_i)_{i\in I} \in \ell_2(I).
\end{equation*}
Then, $(\ell_2(I),\Skindef_2)$ is a Krein space and \eqref{desc fund} is a
fundamental decomposition of $\ell_2(I)$.
Now, if $T:\ell_2(I)\ra\cH$ is the synthesis operator of $\cF$,
its adjoint (in the sense of Krein spaces) is given by
\begin{equation*}
  T^+f=\sum_{i\in I_+} \K{f}{f_i}e_i- \sum_{i\in I_-} \K{f}{f_i}e_i, \qquad f\in\cH.
\end{equation*}

\begin{definition}\label{def:jframe_op}
Given a $J$-frame $\cF=\{f_i\}_{i\in I}$ for $\cH$, the
\emph{$J$-frame operator} $S:\cH\ra\cH$ associated with $\cF$ is defined by
\begin{equation*}
  Sf \defeq TT^+f=\sum_{i\in I_+} \K{f}{f_i}f_i- \sum_{i\in I_-} \K{f}{f_i}f_i, \qquad f\in\cH.
\end{equation*}
\end{definition}

It is easy to see that $S$ is a bounded and boundedly invertible
self-adjoint operator $S$ in the Krein space $\mathcal H$.
We also introduce the positive operators $S_\pm$ by
\begin{equation}\label{eses}
  S_\pm f \defeq \sum_{i\in I_\pm} \K{f}{f_i}f_i, \qquad f\in\cH,
\end{equation}
which are non-negative operators in $\cH$ since
\begin{equation}\label{eses positivos}
  \K{S_\pm f}{f} = \sum_{i\in I_\pm}|\K{f}{f_i}|^2 \qquad \text{for every $f\in \cH$}.
\end{equation}
Clearly, the $J$-frame operator can be written as $S=S_+-S_-$; hence it is
the difference of two positive operators.
The operators $S$ and $S_\pm$ are related via the
projection $Q \defeq \proj{\cM_+}{\cM_-}$ and its
adjoint $Q^+=\proj{\cM_-^{\ort}}{\cM_+^{\ort}}$ as follows:
\begin{equation}\label{eses con q}
  QS = S_+=SQ^+  \qquad \text{and} \qquad (I-Q)S=-S_-=S(I-Q)^+.
\end{equation}
Therefore, we have that $R(S_\pm)=\cM_\pm$ and
\begin{equation}\label{mapping}
  S(\cM_-^{\ort})=\cM_+, \qquad  S(\cM_+^{\ort})=\cM_-.
\end{equation}

Finally, the class of $J$-frame operators can be characterized in the following way;
see \cite[Proposition~5.7]{GMMM12}.

\begin{theorem}\label{oldie}
A bounded and boundedly invertible
self-adjoint operator $S$ in a Krein space $\mathcal H$ is a $J$-frame
operator if and only if the following conditions are satisfied:
\begin{myenum}
\item there exists a maximal uniformly positive subspace
 $\mathcal L_+$ of $\mathcal H$ such that $S(\mathcal L_+)$ is also
 maximal uniformly positive;
\item  $[ Sf , f ] \geq 0$ for every $f\in \mathcal L_+$;
\item  $[ Sg, g ] \leq 0$ for every $g \in (S(\mathcal L_+))^{[\perp]}$.
\end{myenum}
\end{theorem}

\begin{remark}\label{bife}
If $\cF=\{f_i\}_{i\in I}$ is a $J$-frame for $\cH$ with $J$-frame operator $S$
and $\cM_\pm$ are given by \eqref{emes}, then $\cM_-^{\ort}$ is a
maximal uniformly positive subspace satisfying conditions (i)--(iii) in Theorem \ref{oldie}.
In fact, (i) follows from \eqref{mapping}.
Moreover, since $N(S_-)=R(S_-)^{\ort}=\cM_-^{\ort}$,
we have for $f\in \cM_-^{\ort}$
\[
  \K{Sf}{f} = \K{S_+f}{f} = \sum_{i\in I_+}|\K{f}{f_i}|^2 \geq 0;
\]
see \eqref{eses positivos}.
Analogously, $S\big(\cM_-^{\ort}\big)^{\ort}=\cM_+^{\ort}=N(S_+)$ and
if $g\in \cM_+^{\ort}$ then
\[
  \K{Sg}{g} = \K{-S_-g}{g} = -\sum_{i\in I_-}|\K{g}{f_i}|^2 \leq 0.
\]
Thus, we have also shown conditions (ii) and (iii).
\end{remark}

If $S$ is a $J$-frame operator, then $S^{-1}$ is also a $J$-frame operator.
More precisely, the following proposition is true; see \cite[Proposition~5.4]{GMMM12}.

\begin{proposition}\label{canonical dual}
If $\cF=\{f_i\}_{i\in I}$ is a $J$-frame for $\cH$ with $J$-frame operator $S$,
then $\cF'=\{S^{-1}f_i\}_{i\in I}$ is also a $J$-frame for $\cH$. Furthermore,
\[
  \sgn(\K{S^{-1}f_i}{S^{-1}f_i}) = \sgn(\K{f_i}{f_i}) \qquad \text{for every $i\in I$},
\]
the $J$-frame operator of $\cF'$ is $S^{-1}$ and, if $\cM_\pm$ are given by \eqref{emes},
then
\begin{equation*}
  \ol{\linspan\{S^{-1}f_i:\ i\in I_\pm \}} = \cM_\mp^{\ort}.
\end{equation*}
\end{proposition}

\section{Operator matrix representations for $J$-frame operators}\label{sec:matrix}

\noindent
In the following we describe $J$-frame operators via $2\times 2$ block operator matrices.
Let $(\cH,\Skindef)$ be a Krein space with fundamental decomposition
\begin{equation*}
  \cH = \cH_+ \,[\Pluspunkt]\, \cH_-.
\end{equation*}
Every bounded operator $S$ in $\cH$ can be written as a block operator matrix
\begin{equation}\label{bom}
  S = \matriz{A}{B}{C}{D}
\end{equation}
with bounded operators
\[
  A \in L(\cH_+), \quad B \in L(\cH_-,\cH_+), \quad
  C \in L(\cH_+,\cH_-), \quad D \in L(\cH_-).
\]
An operator of the form \eqref{bom} is self-adjoint in the Krein space $\cH$
if and only if $A$ and $D$ are self-adjoint in the Hilbert spaces $(\cH_+,\Skindef)$
and $(\cH_-.-\Skindef)$, respectively, and $C=-B^*$.

\begin{theorem}\label{Irland}
Let $S$ be a bounded self-adjoint operator in a Krein space
$(\mathcal H, \Skindef)$. Then, $S$ is a $J$-frame operator if and only if
there exists a fundamental decomposition
\begin{equation}\label{decomp17}
  \mathcal H = \mathcal H_+ [\Pluspunkt] \mathcal H_-,
\end{equation}
such that $S$ admits a representation with respect to \eqref{decomp17}
of the form
\begin{equation}\label{Cork}
  S= \matriz{A}{-AK}{K^*A}{D}
\end{equation}
where $A$ is a uniformly positive operator in the Hilbert space
$(\mathcal H_+, \Skindef)$, $K: \mathcal H_-\to \mathcal H_+$ is a
uniform contraction (i.e.\ $\|K\|<1$), and $D$ is a self-adjoint operator
such that $D+K^*AK$ is uniformly positive in the
Hilbert space $(\mathcal H_-, -\Skindef)$.
\end{theorem}

\begin{proof}
First, assume that $S$ is a $J$-frame operator and consider
a $J$-frame $\cF=\{f_i\}_{i\in I}$ for $\cH$ with synthesis
operator $T:\ell_2(I)\ra\cH$ such that $S=TT^+$. Let $\cM_\pm$
be as in \eqref{emes}. By \cite[Corollary~1.8.14]{AI89},
the subspace $\cM_-^{\ort}$ is maximal uniformly positive and the
pair $\cM_-^{\ort}$, $\cM_-$ form a fundamental decomposition of $\cH$:
\begin{equation}\label{decomp}
  \mathcal H = \cM_-^{\ort} \,[\Pluspunkt]\, \cM_- .
\end{equation}

Since $S$ is a bounded self-adjoint operator in the
Krein space $(\mathcal H, \Skindef)$, there exist bounded self-adjoint
operators $A$ and $D$ in the Hilbert spaces $(\cM_-^{\ort}, \Skindef)$
and $(\cM_-, -\Skindef)$, respectively, and $B:\cM_- \to \cM_-^{\ort}$ such that
\[
  S = \matriz{A}{B}{-B^*}{D}
\]
with respect to the fundamental decomposition \eqref{decomp}.

We represent the maximal uniformly definite subspaces $\cM_+^{\ort}$ and $\cM_+$
with the help of an angular operator with respect to the fundamental decomposition
\eqref{decomp}, i.e.\ an operator $K:\cM_-\ra \mathcal \cM_-^{\ort}$
with $\|K\|<1$ such that
\begin{equation}\label{Fhafen}
  \cM_+^{\ort} = \left\{ Kx_- + x_- \,\big|\, x_- \in \cM_-\right\}
  \qquad\text{and}\qquad
  \mathcal M_+
  =\{x_+ + K^* x_+ \,|\, x_+ \in  \cM_-^{\ort}\};
\end{equation}
see, e.g.\ \cite[Theorem~1.8.11]{AI89}.

Since $S$ maps $\cM_-^{\ort}$ onto $\mathcal M_+$ (see \eqref{mapping}),
for every $x\in \cM_-^{\ort}$ there exists $x_+ \in \cM_-^{\ort}$ such that
\[
  \matriz{A}{B}{-B^*}{D}
  \vect{x}{0}
  = \vect{Ax}{-B^*x}
  = \vect{x_+}{K^*x_+}
\]
(see \eqref{Fhafen}), i.e.\ $Ax=x_+$ and $-B^*x= K^*Ax$.
Since this holds for all $x\in \cM_-^{\ort}$, we obtain
\[
  B^* = -K^*A \qquad \text{and hence} \qquad B = -AK,
\]
which proves \eqref{Cork}.

The operator $S$ is boundedly invertible (see, e.g.\ \cite[Proposition~5.2]{GMMM12})
and therefore the range of $S$ equals
$\mathcal H= \cM_-^{\ort} [\Pluspunkt] \cM_-$.
From the first row in \eqref{Cork} we conclude that
the range of $A$ equals $\cM_-^{\ort}$.
Since $A$ is self-adjoint, this implies that $A$ is boundedly invertible.

Moreover, by Remark \ref{bife}, the subspace $\cM_-^{\ort}$ satisfies
conditions (i)--(iii) in Theorem~\ref{oldie}.
In particular, condition (ii) in Theorem \ref{oldie} says
that, for $x\in \cM_-^{\ort}$,
\[
  0 \le [Sx,x] = \left[\matriz{A}{-AK}{K^*A}{D}
  \vect{x}{0}, \vect{x}{0}\right] = [Ax,x],
\]
which implies that $A$ is uniformly positive.

Next, let $x_- \in \cM_-$.
Then $Kx_- + x_- \in \mathcal \cM_+^{\ort}$ and
\begin{equation*}
  \matriz{A}{-AK}{K^*A}{D} \vect{Kx_-}{x_-}
  = \vect{0}{(K^*AK+D)x_-}.
\end{equation*}
Since $S(\cM_+^{\ort}) = \cM_-$ (see \eqref{mapping}),
this shows that $K^*AK+D$ is a surjective operator
and hence boundedly invertible (as it is self-adjoint).
Moreover, condition (iii) in Theorem \ref{oldie} states for $x_-\in \cM_-$,
\[
  0 \geq \left[\matriz{A}{-AK}{K^*A}{D} \vect{Kx_-}{x_-}, \vect{Kx_-}{x_-}\right]
  = \bigl[(K^*AK+D)x_-,x_-\bigr],
\]
which proves that $D+K^*AK$ is uniformly positive in $(\cM_-,-\Skindef)$.

\medskip

Conversely, assume that $S$ admits a block operator matrix representation
as in \eqref{Cork}, with respect to the fundamental decomposition \eqref{decomp17}.
Note that $S$ is a boundedly invertible operator because
\[
  S = \matriz{A}{-AK}{K^*A}{D}
  = \matriz{I}{0}{K^*}{I}\matriz{A}{0}{0}{D+K^*AK}\matriz{I}{-K}{0}{I},
\]
and the three block operator matrices on the right-hand side are boundedly invertible.
The space $\mathcal H_+$ is uniformly positive. For $x_+ \in \mathcal H_+$
we have
\[
  \matriz{A}{-AK}{K^*A}{D} \vect{x_+}{0}
  = \vect{Ax_+}{K^*Ax_+}.
\]
Obviously, $S$ maps $\mathcal H_+$ into the uniformly positive
subspace $\{x_+ + K^* x_+ \,|\, x_+ \in  \mathcal H_+ \}$.
Moreover, for $x_+ \in \mathcal H_+$ it follows that
\[
  \left[\matriz{A}{-AK}{K^*A}{D} \vect{x_+}{0}, \vect{x_+}{0}\right]
  = [Ax_+,x_+] \ge 0
\]
since $A$ is positive.  In a similar manner,
using the positivity of $K^*AK +D$ we obtain for
\[
  x\in
  \{x_+ + K^* x_+ \,|\, x_+ \in  \mathcal H_+ \}^{[\perp]}
  = \{Kx_- +  x_- \,|\, x_- \in  \mathcal H_- \}
\]
the estimate
\[
  \left[\matriz{A}{-AK}{K^*A}{D} \vect{Kx_-}{x_-}, \vect{Kx_-}{x_-}\right]
  = [(K^*AK +D)x_-,x_-] \leq 0.
\]
Therefore, by Theorem~\ref{oldie} the operator $S$
is a $J$-frame operator.
\end{proof}

Given a $J$-frame $\cF$ for $\cH$, consider $\cM_\pm$ as in \eqref{emes}.
The block operator representation in Theorem~\ref{Irland} is based
on the fundamental decomposition \eqref{decomp} determined by $\cM_-$.
Obviously, the subspace $\cM_+$ leads in the same way to a
fundamental decomposition:
\begin{equation}\label{decomp2}
  \cH = \cM_+ [\dotplus] \cM_+^{\ort}.
\end{equation}
Therefore, besides \eqref{Cork}, there exists in a natural way a second
block operator representation with respect to \eqref{decomp2}.
Both representations are used in the next section to relate the $J$-frame bounds
with the spectrum of the $J$-frame operator.

\begin{theorem}\label{Scotland}
Let $S$ be a bounded self-adjoint operator in a Krein space
$(\mathcal H, \Skindef)$. Then, $S$ is a $J$-frame operator if and only if
there exists a fundamental decomposition
\begin{equation}\label{decomp18}
  \mathcal H = \mathcal K_+ [\Pluspunkt] \mathcal K_-,
\end{equation}
such that $S$ admits a representation with respect to \eqref{decomp18} of the form
\begin{equation}\label{Edimburg}
  S = \matriz{A'}{LD'}{-D'L^*}{D'}
\end{equation}
where $A'$ is a self-adjoint operator such that $A'+LD'L^*$ is a
uniformly positive operator in the Hilbert space $(\mathcal K_+, \Skindef)$,
$L: \mathcal K_-\to \mathcal K_+$ is a uniform contraction (i.e.\ $\|L\|<1$),
and $D'$ is uniformly positive in the Hilbert space $(\mathcal K_-, -\Skindef)$.
\end{theorem}

\begin{proof}
Assume that $\cF=\{f_i\}_{i\in I}$ is a $J$-frame for $\cH$ with $J$-frame operator $S$.
Consider the maximal uniformly definite subspaces $\cM_\pm$ given
by \eqref{emes}, the fundamental decomposition
\[
  \cH = \cM_+ [\dotplus] \cM_+^{\ort},
\]
and the angular operator $L:\cM_+^{\ort} \ra \cM_+$ associated with $\cM_-$.
Since $S$ maps $\cM_+^{\ort}$ onto $\cM_-$ (see \eqref{mapping}),
the representation \eqref{Edimburg} follows in the same way as the
representation \eqref{Cork} was proved in Theorem \ref{Irland}.
The converse also follows the same ideas used in Theorem~\ref{Irland}.
\end{proof}

As it was mentioned before, if $\cF$ is a $J$-frame for $\cH$,
the subspaces $\cM_\pm$ given in \eqref{emes} are maximal
uniformly definite subspaces with opposite signature and $\cH$
can be decomposed as $\cH=\cM_+ \dotplus \cM_-$; see \eqref{suma}.
Observe that the projection $Q=\proj{\cM_+}{\cM_-}$ has
the following block operator matrix representation
with respect to the fundamental decomposition \eqref{decomp}:
\begin{equation}\label{panadaria}
  Q = \matriz{I}{0}{K^*}{0},
\end{equation}
where $K:\cM_-\ra\cM_-^{\ort}$ is the angular operator associated with $\cM_+^{\ort}$.
On the other hand, with respect to the fundamental decomposition \eqref{decomp2},
$Q$ is represented as
\begin{equation}\label{panadaria2}
  Q = \matriz{I}{-L}{0}{0},
\end{equation}
where $L:\cM_+^{\ort}\ra\cM_+$ is the angular operator associated with $\cM_-$.

\begin{corollary}\label{eses pm}
Let $\cF=\{f_i\}_{i\in I}$ be a $J$-frame for $\cH$ with $J$-frame operator $S$,
let $S_\pm$ be as in \eqref{eses}, and let $\cM_\pm$ be as in \eqref{emes}.
Then the following statements are true.
\begin{myenum}
\item
With respect to the fundamental decomposition \eqref{decomp},
the positive operators $S_+$ and $S_-$ are represented as
\begin{equation}\label{lomo}
  S_+ = \matriz{A}{-AK}{K^*A}{-K^*AK} \qquad \text{and} \qquad
  S_- = \matriz{0}{0}{0}{-(D+K^*AK)}.
\end{equation}
\item
With respect to the fundamental decomposition \eqref{decomp2},
the positive operators $S_+$ and $S_-$ are represented as
\begin{equation}\label{lomito}
  S_+ = \matriz{A'+LD'L^*}{0}{0}{0} \qquad \text{and} \qquad
  S_- = \matriz{LD'L^*}{-LD'}{D'L^*}{-D'}.
\end{equation}
\end{myenum}
\end{corollary}

\begin{proof}
The positive operators $S_+$ and $S_-$ can be written as $S_+=QS$ and $S_-=-(I-Q)S$,
where $Q=\proj{\cM_+}{\cM_-}$; see \eqref{eses con q}.
Then the block operator matrix representations of $S_+$ and $S_-$
follow easily by multiplying \eqref{panadaria} with \eqref{Cork},
and \eqref{panadaria2} with \eqref{Edimburg}, respectively.
\end{proof}

Finally, if $S$ is a $J$-frame operator, we derive block operator matrix
representations for its inverse $S^{-1}$.

\begin{theorem}\label{Sminus1}
Let $S$ be a bounded and boundedly invertible self-adjoint operator acting on a
Krein space $(\cH,\Skindef)$.  Then, the following conditions are equivalent:
\begin{myenum}
\item
$S$ is a $J$-frame operator;
\item
there exists a fundamental decomposition $\cH=\cH_+ \sdo \cH_-$
such that $S^{-1}$ is represented as
\begin{equation}\label{Sinv}
  S^{-1} = \matriz{A^{-1}- KZK^*}{KZ}{-ZK^*}{Z},
\end{equation}
where $A$ is a uniformly positive operator in the
Hilbert space $(\mathcal H_+, \Skindef)$, $Z$  is a uniformly positive operator
in the Hilbert space $(\mathcal H_-, -\Skindef)$
and $K: \mathcal H_-\to \mathcal H_+$ is a uniform contraction;
\item
there exists a fundamental decomposition $\cH=\cK_+ \sdo \cK_-$
such that $S^{-1}$ is represented as
\begin{equation}\label{Fritzy}
  S^{-1} = \matriz{Y}{-YL}{L^*Y}{(D')^{-1}-L^*YL},
\end{equation}
where $D'$ is a uniformly positive operator in the
Hilbert space $(\cK_-, -\Skindef)$, $Y$ is a uniformly positive operator
in the Hilbert space $(\cK_+, \Skindef)$ and $L: \cK_-\to \cK_+$ is a uniform contraction.
\end{myenum}
\end{theorem}

\begin{proof}
First, we prove the equivalence (i) $\Leftrightarrow$ (ii).
By Theorem~\ref{Irland}, $S$ is a $J$-frame operator if and only if
there exists a fundamental decomposition $\cH=\cH_+ \sdo \cH_-$
such that $S$ is represented as \eqref{Cork}, where $A$ is a
uniformly positive operator in the Hilbert space $(\mathcal H_+, \Skindef)$,
$D$ is a self-adjoint operator in the Hilbert space $(\mathcal H_-, -\Skindef)$
such that $D+K^*AK$ is uniformly positive, and $K: \mathcal H_-\to \mathcal H_+$
with $\|K\|<1$.

Observe that the representation of $S$ in \eqref{Cork} can be written as
\[
  S = \matriz{I}{0}{K^*}{I} \matriz{A}{0}{0}{D+K^*AK} \matriz{I}{-K}{0}{I}.
\]
Hence, if $Z\defeq (D+K^*AK)^{-1}$, $S$ is a $J$-frame operator if and only if
\begin{align*}
  S^{-1} = \matriz{I}{K}{0}{I} \matriz{A^{-1}}{0}{0}{Z} \matriz{I}{0}{-K^*}{I}
  = \matriz{A^{-1}-KZK^*}{KZ}{-ZK^*}{Z},
\end{align*}
where $A$ and $Z$ are uniformly positive operators and $K$ is a uniform contraction.

The proof of (i) $\Leftrightarrow$ (iii) is similar, where one uses that $S$ can be represented
as in \eqref{Edimburg} in Theorem~\ref{Scotland} and one defines $Y\defeq (A'+LD'L^*)^{-1}$.
\end{proof}

\begin{remark}\label{cuadril}
It follows from Proposition~\ref{canonical dual} that
if $S:\cH\ra\cH$ is a $J$-frame operator, then $S^{-1}$ is also a $J$-frame operator.
Moreover, the representation of $S^{-1}$ in \eqref{Sinv} uses the components
of the representation of $S$ in \eqref{Cork};
note that, according to the proof, $Z=(D+K^*AK)^{-1}$.
However, \eqref{Sinv} is of the form as in \eqref{Edimburg} in Theorem~\ref{Scotland}
because with the notation from that theorem one has that $D'=Z$ is uniformly positive,
$L=K$ is a uniform contraction, $ A'= A^{-1}-KZK^*$ satisfies
\[
  A'+LD'L^* = (A^{-1}-KZK^*)+KZK^* = A^{-1},
\]
and, hence, $A'+LD'L^*$ is uniformly positive.

Similarly, the representation of $S^{-1}$ in \eqref{Fritzy} uses the components
of the representation of $S$ in \eqref{Edimburg} with $Y=(A'+LD'L^*)^{-1}$,
and the representation of $S^{-1}$ is of the form as in \eqref{Cork}.
\end{remark}

\section{Interpreting the $J$-frame bounds as \\ spectral bounds of the $J$-frame operator}
\label{sec:framebounds}

\noindent
Let $\cF=\{f_i\}_{i\in I}$ be a $J$-frame for $\cH$ and let $\cM_\pm$ be
as in \eqref{emes}.
By Theorem~\ref{thm J frame bounds}, $\cF_\pm=\{f_i\}_{i\in I_\pm}$ is a
frame for $(\cM_\pm, \pm\Skindef)$,
i.e.\ there exist $0<\alpha_\pm \leq \beta_\pm$ such that
\[
  \alpha_\pm (\pm\K{f}{f}) \leq \sum_{i\in I_\pm} \big|\K{f}{f_i}\big|^2
  \leq \beta_\pm (\pm\K{f}{f}) \qquad \text{for every $f\in \cM_\pm$},
\]
cf.\ \eqref{eq J frame bounds}.
The frame bounds of $\cF_\pm$ (i.e.\ the optimal set of
constants $0<\alpha_\pm \leq \beta_\pm$) are called the \emph{$J$-frame bounds of $\cF$}.

It is our aim to recover the $J$-frame bounds of $\cF$ as spectral bounds
for the $J$-frame operator $S$.  Moreover, we recover also the $J$-frame bounds
of the dual $J$-frame $\cF'=\{S^{-1}f_i\}_{i\in I}$.
This is achieved in a straightforward approach by using the $2\times 2$
block operator matrix representations for $S$, $S_\pm$ and $S^{-1}$
from Section~\ref{sec:matrix}.

\begin{proposition}\label{pr:Jframebounds}
 Let $0<\alpha_\pm \leq \beta_\pm$ be the $J$-frame bounds of the $J$-frame $\cF$.
 If the $J$-frame operator $S$ is represented as in \eqref{Cork}, then
\[
\alpha_-=\inf\, W(D+K^*AK) \qquad \text{and} \qquad \beta_-=\sup\, W(D+K^*AK).
\]
If the $J$-frame operator $S$ is represented as in \eqref{Edimburg}, then
\[
\alpha_+=\inf\, W(A'+LD'L^*) \qquad \text{and} \qquad \beta_+=\sup\, W(A'+LD'L^*).
\]
\end{proposition}

\begin{proof}
If $f\in \cM_-$ then, by \eqref{eses positivos} and \eqref{lomo},
\begin{align*}
  \sum_{i\in I_-} \big|\K{f}{f_i}\big|^2 &= \K{S_- f}{f}
	=-\K{\matriz{0}{0}{0}{D+K^*AK}\vect{0}{f}}{\vect{0}{f}} \\
  &= -\K{(D+K^*AK)f}{f}.
\end{align*}
As $\cF_-=\{f_i\}_{i\in I_-}$ is a frame for $(\cM_-,-\Skindef)$ with
frame bounds $0<\alpha_- \leq \beta_-$,
it is immediate that $\inf\, W(D+K^*AK)= \alpha_-$ and $\sup\, W(D+K^*AK)= \beta_-$.

Similarly, if $f\in \cM_+$ then, by \eqref{eses positivos} and \eqref{lomito},
\begin{align*}
  \sum_{i\in I_+} \big|\K{f}{f_i}\big|^2 &= \K{S_+ f}{f}
	=\K{\matriz{A'+LD'L^*}{0}{0}{0}\vect{f}{0}}{\vect{f}{0}} \\
  &= \K{(A'+LD'L^*)f}{f}.
\end{align*}
Since $\cF_+=\{f_i\}_{i\in I_+}$ is a frame for $(\cM_+,\Skindef)$ with
frame bounds $0<\alpha_+ \leq \beta_+$, it follows
that $\inf\, W(A'+LD'L^*)= \alpha_+$ and $\sup\, W(A'+LD'L^*)= \beta_+$.
\end{proof}

In what follows, we derive formulae for the $J$-frame bounds of the
canonical dual $J$-frame $\cF'=\{S^{-1}f_i\}_{i\in I}$.
Recall that it is also a $J$-frame for $\cH$,
and its $J$-frame operator is $S^{-1}$; see Proposition~\ref{canonical dual}.
In addition, $\cF'_\pm=\{S^{-1}f_i\}_{i\in I_\pm}$ is a frame
for $(\cM_\mp^{\ort}, \pm\Skindef)$,
i.e.\ there exist constants $0<\gamma_\pm \leq \delta_\pm$ such that
\[
  \gamma_\pm (\pm\K{f}{f}) \leq \sum_{i\in I_\pm} \big|\K{f}{S^{-1}f_i}\big|^2
  \leq \delta_\pm (\pm\K{f}{f}) \qquad \text{for every $f\in \cM_\mp^{\ort}$}.
\]
The optimal set of constants $0<\gamma_\pm \leq \delta_\pm$
are the $J$-frame bounds of $\cF'$.

\begin{proposition}\label{pr:dualJframebounds}
Let $\cF$ be a $J$-frame for $\cH$ with $J$-frame operator $S$, and
let $0<\gamma_\pm \leq \delta_\pm$ be the $J$-frame bounds of the
canonical dual $J$-frame $\cF'$.
If the $J$-frame operator $S$ is represented as in \eqref{Cork}, then
\[
  \gamma_+ = \inf\, W(A^{-1})  \qquad \text{and} \qquad	\delta_+ = \sup\, W(A^{-1}).
\]
If the $J$-frame operator $S$ is represented as in \eqref{Edimburg}, then
\[
  \gamma_- = \inf\, W((D')^{-1})  \qquad \text{and} \qquad \delta_- = \sup\, W((D')^{-1}).
\]
\end{proposition}

\begin{proof}
The $J$-frame operator $S^{-1}$ can be written as $S^{-1}=(S^{-1})_+ - (S^{-1})_-$,
where $(S^{-1})_+=Q^+S^{-1}$ and $(S^{-1})_-=-(I-Q)^+S^{-1}$ are positive operators
in the Krein space $\cH$; see \eqref{eses con q}.

Also, if $S$ is represented as in \eqref{Cork} then, by Remark \ref{cuadril}, $S^{-1}$
is represented as in \eqref{Sinv}. Moreover, following the same arguments as in
Corollary~\ref{eses pm} it can be proved that
\[
  (S^{-1})_+ = \matriz{A^{-1}}{0}{0}{0} \qquad \text{and} \qquad
  (S^{-1})_-=\matriz{KZK^*}{-KZ}{ZK^*}{-Z}.
\]

If $f\in \cM_-^{\ort}$ then, by \eqref{eses positivos} and the
above representation for $(S^{-1})_+$,
\begin{align*}
  \sum_{i\in I_+}|\K{f}{S^{-1}f_i}|^2 &=\K{(S^{-1})_+f}{f}
  =\K{\matriz{A^{-1}}{0}{0}{0}\vect{f}{0}}{\vect{f}{0}} \\
  &=\K{A^{-1}f}{f}.
\end{align*}
As $\{S^{-1}f_i\}_{i\in I_+}$ is a frame for $(\cM_-^{\ort},\Skindef)$ with
frame bounds $0<\gamma_+ \leq \delta_+$, it follows that $\inf\, W(A^{-1})= \gamma_+$
and $\sup\, W(A^{-1})= \delta_+$.

On the other hand, if the $J$-frame operator $S$ is represented as in \eqref{Edimburg}
then, by Remark~\ref{cuadril}, $S^{-1}$ is represented as in \eqref{Fritzy}. Then,
\[
  (S^{-1})_+=\matriz{Y}{-YL}{L^*Y}{-L^*YL} \qquad \text{and} \qquad
  (S^{-1})_-=\matriz{0}{0}{0}{-(D')^{-1}},
\]
and, since $\{S^{-1}f_i\}_{i\in I_-}$ is a frame for $(\cM_+^{\ort},-\Skindef)$ with
frame bounds $0<\gamma_-\leq \delta_-$, the characterization
of $\gamma_-$ and $\delta_-$ follows.
\end{proof}

\section{Spectrum of the $J$-frame operator}\label{sec:spectrum}

\noindent
In Theorem~\ref{th:spec_incl} below we describe the location of the spectrum of $S$.
In its proof we use the first and second Schur complements of $S$
according to the representations of $S$ given in Theorems~\ref{Scotland}
and \ref{Irland}, respectively.

\medskip

Let $S$ be a $J$-frame operator in a Krein space $(\cH,\Skindef)$
and assume first that $\cH=\cH_+\sdo\cH_-$ is a fundamental decomposition
such that $S$ is decomposed as in \eqref{Cork} in Theorem~\ref{Irland}.
In this case we use the second Schur complement of $S$, which is defined as
\[
  S_2(\lambda) \defeq D-\lambda+K^*A(A-\lambda)^{-1}AK, \qquad \lambda\in\rho(A).
\]
On $\rho(A)$ the spectra of the operator function $S_2(\lambda)$ and
the operator $S$ coincide, i.e.\
\begin{equation}\label{spec_eqSS2}
  0 \in \sigma(S_2(\lambda)) \quad\iff\quad \lambda \in \sigma(S)
  \qquad\text{for}\;\; \lambda\in\rho(A);
\end{equation}
see, e.g.\ \cite[Proposition~1.6.2]{Tretter08}.

Since $A$ and $D+K^*AK$ are uniformly positive in the
Hilbert spaces $(\mathcal{H}_+,\Skindef)$  and $(\mathcal{H}_-,-\Skindef)$, respectively,
their numerical ranges $W(A)$ and $W(D+K^*AK)$ are intervals in $(0,+\infty)$.
Also $D$ is self-adjoint in the Hilbert space $(\mathcal{H}_-,-\Skindef)$;
so its numerical range $W(D)$ is a real interval.

On the other hand, if $\cH=\cK_+\sdo\cK_-$ is a fundamental decomposition
of $\mathcal{H}$ such that $S$ is decomposed as in \eqref{Edimburg} in
Theorem~\ref{Scotland}, we use the first Schur complement of $S$,
which is defined as
\[
  S_1(\lambda) \defeq A'-\lambda+LD'(D'-\lambda)^{-1}D'L^*, \qquad \lambda\in\rho(D').
\]
On $\rho(D')$ the spectra of the operator function $S_1(\lambda)$ and
the operator $S$ coincide, i.e.\
\begin{equation*}
  0 \in \sigma(S_1(\lambda)) \quad\iff\quad \lambda \in \sigma(S)
  \qquad\text{for}\;\; \lambda\in\rho(D');
\end{equation*}
see, e.g.\ \cite[Proposition~1.6.2]{Tretter08}.
In this case the numerical ranges $W(D')$ and $W(A'+LD'L^*)$ are
intervals in $(0,+\infty)$ and $W(A')$ is a real interval.

\begin{theorem}\label{th:spec_incl}
Let $S$ be a $J$-frame operator in the Krein space $(\mathcal H, \Skindef)$.
\begin{myenum}
\item
If $S$ is decomposed as in \eqref{Cork}, assume
that $\ol{W(A)}=[a_-,a_+]$, $\ol{W(D)}=[d_-,d_+]$, and $\ol{W(D+K^*AK)}=[b_-,b_+]$.
Then, the spectrum of $S$ satisfies the following inclusions:
\begin{align}
  \sigma(S)\setminus\dR &\subseteq \biggl\{\lambda\in\dC \,:\, |\lambda-a_+|<a_+
  \;\;\text{and}\;\;  \RE\lambda>\frac{b_-}{2}\biggr\},
    \label{spec_incl1}\\[1ex]
  \sigma(S)\cap\dR &\subseteq \bigl[\min\{a_-,b_-\},\max\{a_+,d_+\}\bigr].
    \label{spec_incl2}
\end{align}
\item
If $S$ is decomposed as in \eqref{Edimburg}, assume
that $\ol{W(A')}=[a'_-,a'_+]$, $\ol{W(D')}=[d'_-,d'_+]$, and $\ol{W(A'+LD'L^*)}=[b'_-,b'_+]$.
Then, the spectrum of $S$ satisfies the following inclusions:
\begin{align}
  \sigma(S)\setminus\dR &\subseteq \biggl\{\lambda\in\dC \,:\, |\lambda-d'_+|<d'_+
  \;\;\text{and}\;\;  \RE\lambda>\frac{b'_-}{2}\biggr\}, \label{spec_incl3}
    \\[1ex]\nonumber
  \sigma(S)\cap\dR &\subseteq \bigl[\min\{d'_-,b'_-\},\max\{a'_+,d'_+\}\bigr].
\end{align}
\end{myenum}
\end{theorem}

\begin{proof}
To prove (i), let us rewrite $S_2(\lambda)$ as follows:
\begin{align*}
  S_2(\lambda) &= D-\lambda+K^*A(A-\lambda)^{-1}(A-\lambda+\lambda)K \\[1ex]
  &= D-\lambda+K^*AK+\lambda K^*A(A-\lambda)^{-1}K \\[1ex]
  &= D+K^*AK-\lambda(I-K^*K)+\lambda^2K^*(A-\lambda)^{-1}K.
\end{align*}
Hence, for $g\in\cH_-$ with $\|g\|=1$ we obtain
\begin{align}
  \bigl\langle S_2(\lambda)g,g\bigr\rangle
  &= \bigl\langle (D+K^*AK)g,g\bigr\rangle-\lambda\bigl(1-\|Kg\|^2\bigr)
  +\lambda^2\bigl\langle(A-\lambda)^{-1}Kg,Kg\bigr\rangle
    \label{formS2_1}\\[1ex]
  &= \bigl\langle (D+K^*AK)g,g\bigr\rangle-\lambda\bigl(1-\|Kg\|^2\bigr)
  	\notag\\[1ex]
  &\quad+ \lambda^2\bigl\langle A(A-\lambda)^{-1}Kg,(A-\lambda)^{-1}Kg\bigr\rangle-\lambda|\lambda|^2\bigl\|(A-\lambda)^{-1}Kg\bigr\|^2.
    \label{formS2_2}
\end{align}
First, we show that
\begin{equation}\label{halfline}
  \bigl(-\infty,\min\{a_-,b_-\}\bigr) \subseteq \rho(S).
\end{equation}
Let $\lambda\in \bigl(-\infty,\min\{a_-,b_-\}\bigr)$.
Then $\lambda\in\rho(A)$ and $(A-\lambda)^{-1}$ is positive.
Therefore \eqref{formS2_1} yields
\begin{align*}
  \bigl\langle S_2(\lambda)g,g\bigr\rangle
  &\ge \bigl\langle(D+K^*AK)g,g\bigr\rangle - \lambda\bigl(1-\|Kg\|^2\bigr)
    \\[0.5ex]
  &\ge b_- - \lambda\bigl(1-\|Kg\|^2\bigr)
    \ge b_- - \max\{\lambda,0\}\bigl(1-\|Kg\|^2\bigr)
    \\[0.5ex]
  &\ge b_- - \max\{\lambda,0\} > 0
\end{align*}
for every $g\in\cH_-$ with $\|g\|=1$.
This implies that $0$ is not in the closure of the numerical range of $S_2(\lambda)$
and hence $0\in\rho(S_2(\lambda))$.
By \eqref{spec_eqSS2} this shows that $\lambda\in\rho(S)$.
Therefore \eqref{halfline} is proved.

In \cite[Theorem~2.1]{LLMT05} it was shown that
$\sigma(S)\cap\dR \subseteq \bigl[\min\{a_-,d_-\},\max\{a_+,d_+\}\bigr]$.
Since $A$ is a uniformly positive operator,
$d_-=\min\, \ol{W(D)}\leq \min\, \ol{W(D+K^*AK)}=b_-$.
This, together with \eqref{halfline} shows \eqref{spec_incl2}.

\medskip

Next, we show that
\begin{equation}\label{spec_incl_pr1}
  \sigma(S)\setminus\dR \subseteq \bigl\{\lambda\in\dC\setminus\dR \,:\, |\lambda-a_+|<a_+\}.
\end{equation}
Let $\lambda=x+iy$ be in the complement of the right-hand side of \eqref{spec_incl_pr1}
and assume that $\lambda$ is non-real, i.e.\ $y\ne0$ and
\begin{equation}\label{disc}
  x^2-2a_+x+y^2 \ge 0.
\end{equation}
Since the spectrum of $S$ is symmetric with respect to the
real axis, we may assume, without loss of generality, that $y>0$.
Consider the spectral function $E_t$ associated with the (positive) operator $A$:
\[
  A = \int_{a_-}^{a_+} t\; \rd E_t.
\]
Then, we can rewrite some terms that appear in \eqref{formS2_2} in
terms of $E_t$:
\begin{align*}
  \bigl\langle A(A-\lambda)^{-1}Kg,(A-\lambda)^{-1}Kg\bigr\rangle
  &= \int_{a_-}^{a_+} \frac{t}{|t-\lambda|^2}\,\rd \langle E_t Kg,Kg\rangle, \\
  \bigl\|(A-\lambda)^{-1}Kg\bigr\|^2
  &= \int_{a_-}^{a_+} \frac{1}{|t-\lambda|^2}\,\rd \langle E_t Kg,Kg\rangle.
\end{align*}
For $g\in\cH_-$ with $\|g\|=1$ we obtain from \eqref{formS2_2} that
\begin{align}
  \IM\bigl\langle S_2(\lambda)g,g\bigr\rangle
  &= -y\bigl(1-\|Kg\|^2\bigr) + 2xy\bigl\langle A(A-\lambda)^{-1}Kg,(A-\lambda)^{-1}Kg\rangle
    \notag\\[0.5ex]
  &\quad - y(x^2+y^2)\bigl\|(A-\lambda)^{-1}Kg\bigr\|^2
    \notag\\[0.5ex]
  &= -y\bigl(1-\|Kg\|^2\bigr) + y\int_{a_-}^{a_+}\frac{2xt-(x^2+y^2)}{|t-\lambda|^2}\,
  \rd\langle E_t Kg,Kg\rangle.
    \label{rhsImS2}
\end{align}
If $x\le0$, then the numerator of the fraction in the integral in \eqref{rhsImS2}
is negative as $t>0$;
if $x>0$, then \eqref{disc} yields
\[
  2xt-(x^2+y^2) \le 2xa_+-(x^2+y^2) \le 0.
\]
In both cases the integral in \eqref{rhsImS2} is non-positive and hence
\[
  \IM\langle S_2(\lambda)g,g\bigr\rangle \le -y\bigl(1-\|Kg\|^2\bigr)
  \le -y\bigl(1-\|K\|^2\bigr) < 0.
\]
This shows that $0$ is not in the closure of the numerical range of $S_2(\lambda)$
for such $\lambda$ and hence $0\in\rho(S_2(\lambda))$.
Again by \eqref{spec_eqSS2} this shows that $\lambda\in\rho(S)$.
Thus \eqref{spec_incl_pr1} is proved.

\medskip

It remains to show that
$\sigma(S)\setminus\dR \subseteq \bigl\{\lambda\in\dC\setminus\dR \,:\, \RE\lambda>\frac{b_-}{2}\}$,
but we postpone this until we have completed the proof of (ii).

\medskip

To prove (ii), note that
$S_1(\lambda)=A' + LD'L^* - \lambda(I-LL^*)+ \lambda^2 L(D'-\lambda)^{-1}L^*$
for $\lambda\in \rho(D')$.  Following the same arguments as above, it follows that
\[
  \sigma(S)\cap\dR \subseteq \bigl[\min\{d'_-,b'_-\},\max\{a'_+,d'_+\}\bigr]
\]
and
\begin{equation}\label{spec_incl_pr11}
  \sigma(S)\setminus\dR
  \subseteq \left\{\lambda\in\dC\setminus\dR \,:\, |\lambda-d'_+|<d'_+ \right\}.
\end{equation}
It follows from Remark~\ref{cuadril} that if $S$ is represented as in \eqref{Edimburg}
then the representation \eqref{Fritzy} corresponds to \eqref{Cork}
for $S^{-1}$.  Hence, applying \eqref{spec_incl_pr1} to the $J$-frame operator $S^{-1}$
we obtain that
\[
  \sigma(S^{-1})\setminus\dR
  \subseteq \left\{\lambda\in\dC\setminus\dR \,:\,
  \left|\lambda-\frac{1}{b'_-}\right|<\frac{1}{b'_-}\right\},
\]
because $Y=(A'+LD'L^*)^{-1}$ and
$\sup\, W((A'+LD'L^*)^{-1})=(\inf\, W(A'+LD'L^*))^{-1}=\frac{1}{b'_-}$.
Therefore,
\[
  \sigma(S)\setminus\dR
  = \Bigl\{\lambda\in\dC\setminus\dR \,:\, \frac{1}{\lambda}\in\sigma(S^{-1})\Bigr\}
  \subseteq \Bigl\{\lambda\in\dC\setminus\dR \,:\, \RE\lambda>\frac{b'_-}{2}\Bigr\}.
\]
This, together with \eqref{spec_incl_pr11}, yields \eqref{spec_incl3} and (ii) is shown.

\medskip

Finally, to conclude the proof of (i), note that if $S$ is represented
as in \eqref{Cork} then the representation \eqref{Sinv} corresponds
to \eqref{Edimburg} for $S^{-1}$ (see Remark~\ref{cuadril}).
Applying \eqref{spec_incl_pr11} to the $J$-frame operator $S^{-1}$ we obtain that
\[
  \sigma(S^{-1})\setminus\dR
  \subseteq \left\{\lambda\in\dC\setminus\dR \,:\,
  \left|\lambda-\frac{1}{b_-}\right|<\frac{1}{b_-}\right\},
\]
because $Z=(D+K^*AK)^{-1}$ and
$\sup\, W((D+K^*AK)^{-1})=(\inf\, W(D+K^*AK))^{-1}=\frac{1}{b_-}$.
Hence,
\[
  \sigma(S)\setminus\dR
  = \Bigl\{\lambda\in\dC\setminus\dR \,:\, \frac{1}{\lambda}\in\sigma(S^{-1})\Bigr\}
  \subseteq \Bigl\{\lambda\in\dC\setminus\dR \,:\, \RE\lambda>\frac{b_-}{2}\Bigr\}.
\]
This, together with \eqref{spec_incl_pr1}, yields \eqref{spec_incl1}.
\end{proof}

\begin{remark}
Let $S$ be a $J$-frame operator in $(\cH,\Skindef)$ and assume that $S$ is
represented as in \eqref{Cork}.  By \cite[Theorem~2.1]{LLMT05}, we have
another enclosure for the non-real spectrum of $S$:
\[
  |\IM z| \leq \|AK\| \qquad \text{for $z\in\sigma(S)\setminus \mathbb{R}$}.
\]
Therefore, if $K=0$ then the spectrum of $S$ is contained in $(0,+\infty)$.
Analogously, if $S$ is represented as in \eqref{Edimburg}
then $|\IM z|\leq \|LD'\|$ for every $z\in\sigma(S)\setminus \mathbb{R}$.
Thus, $L=0$ implies that $\sigma(S)\subseteq (0,+\infty)$.

Assume that $\cF=\{f_i\}_{i\in I}$ is a $J$-frame for $\cH$
with $J$-frame operator $S$, and $\cM_\pm$ are defined as in \eqref{emes}.
Then $K$ and $L$ are the angular operators of $\cM_+^{\ort}$ and $\cM_-$
according to the decompositions \eqref{decomp} and \eqref{decomp2}, respectively.
Hence, if $\cM_-=\cM_+^{\ort}$ we have that $K=L=0$ and $\sigma(S)\subseteq (0,+\infty)$.
Moreover, in this case $S$ is a uniformly positive operator in $(\cH,\Skdef)$.
\end{remark}

Given a $J$-frame $\cF=\{f_i\}_{i\in I}$ with $J$-frame operator $S$,
we can use Theorem~\ref{th:spec_incl} to give a spectral enclosure
for $\sigma(S)$ in terms of the $J$-frame bounds of $\cF$ and
its canonical dual $J$-frame $\cF'=\{S^{-1}f_i\}_{i\in I}$.

\begin{corollary}
Let $\cF=\{f_i\}_{i\in I}$ be a $J$-frame for a Krein space $(\cH, \Skindef)$
with $J$-frame operator $S$ and $J$-frame bounds $0<\alpha_\pm\leq \beta_\pm$.
Assume also that $0<\gamma_\pm\leq\delta_\pm$ are the $J$-frame bounds of
its canonical dual frame $\cF'=\{S^{-1}f_i\}_{i\in I}$.
Then, the spectrum of $S$ satisfies the following inclusions:
\begin{align}
  & \sigma(S)\cap\dR \subseteq \left[\varepsilon_-,\varepsilon_+\right],
  \label{real_minimal_spec}
  \\
  & \sigma(S)\setminus\dR
  \subseteq \left\{\lambda\in\dC \,:\, \RE \lambda > \tfrac{\alpha}{2},\;\;
  |\lambda- \tfrac{1}{\gamma}|< \tfrac{1}{\gamma}\right\},
  \label{minimal_spec}
\end{align}
where
\begin{align*}
  \varepsilon_- &= \max\left\{\min\left\{\frac{1}{\delta_+},\alpha_-\right\},
  \min\left\{\frac{1}{\delta_-},\alpha_+\right\}\right\},
  \\[1ex]
  \varepsilon_+ &= \min\left\{\max\left\{\frac{1}{\gamma_+},\beta_-\right\},
  \max\left\{\frac{1}{\gamma_-},\beta_+\right\}\right\},
  \\[1ex]
  \alpha &= \max\{\alpha_+,\alpha_-\} \qquad\text{and}\qquad
  \gamma=\max\{\gamma_+,\gamma_-\}.
\end{align*}
\end{corollary}

\begin{proof}
Given a $J$-frame $\cF=\{f_i\}_{i\in I}$ for $\cH$ with $J$-frame operator $S$,
let $\cM_\pm$ be given by \eqref{emes}. By Theorem \ref{Irland}, according to
the fundamental decomposition $\cH=\cM_-^{\ort}\sdo \cM_-$, $S$ can be represented
as in \eqref{Cork}. Also, according to the fundamental
decomposition $\cH=\cM_+\sdo \cM_+^{\ort}$, $S$ can be represented as in \eqref{Edimburg}.
By Propositions \ref{pr:Jframebounds} and \ref{pr:dualJframebounds},
\begin{alignat*}{3}
  \ol{W(A)} &= \left[\frac{1}{\delta_+},\frac{1}{\gamma_+}\right] \qquad & &\text{and} & \qquad
  \ol{W(D+K^*AK)} &= [\alpha_-,\beta_-],
  \\[1ex]
  \ol{W(D')} &= \left[\frac{1}{\delta_-},\frac{1}{\gamma_-}\right] \qquad & &\text{and} & \qquad
  \ol{W(A'+LD'L^*)} &= [\alpha_+,\beta_+].
\end{alignat*}
Since $A$ is uniformly positive, note that $\inf\, W(D)\leq \alpha_-$
and $\sup\, W(D)\leq \beta_-$. Analogously, the uniform positiveness of $D'$
implies that $\inf\, W(A')\leq \alpha_+$ and $\sup\, W(A')\leq \beta_+$.
Hence, by Theorem \ref{th:spec_incl}, we have that
\begin{align*}
  \sigma(S)\cap\dR &\subseteq \left[\min\left\{\frac{1}{\delta_+},\alpha_-\right\},
  \max\left\{\frac{1}{\gamma_+},\beta_-\right\}\right], \quad
  \\
  \sigma(S)\cap\dR &\subseteq \left[\min\left\{\frac{1}{\delta_-},\alpha_+\right\},
  \max\left\{\frac{1}{\gamma_-},\beta_+\right\}\right],
\end{align*}
and \eqref{real_minimal_spec} follows by intersecting the above intervals.
For the non-real part of the spectrum of $S$, Theorem \ref{th:spec_incl} yields
\begin{align*}
  \sigma(S)\setminus\dR &\subseteq \biggl\{\lambda\in\dC \,:\,
  \biggl|\lambda-\frac{1}{\gamma_+}\biggr|<\frac{1}{\gamma_+}
  \;\;\text{and}\;\;  \RE\lambda>\frac{\alpha_-}{2}\biggr\}, \quad
  \\[1ex]
  \sigma(S)\setminus\dR &\subseteq \biggl\{\lambda\in\dC \,:\,
  \biggl|\lambda-\frac{1}{\gamma_-}\biggr|<\frac{1}{\gamma_-}
  \;\;\text{and}\;\;  \RE\lambda>\frac{\alpha_+}{2}\biggr\},
\end{align*}
and \eqref{minimal_spec} follows by intersecting the two enclosures above.
\end{proof}

\begin{figure}[H]
\definecolor{ffzztt}{rgb}{1.,0.6,0.2}
\definecolor{uuuuuu}{rgb}{0.26666666666666666,0.26666666666666666,0.26666666666666666}
\definecolor{qqqqff}{rgb}{0.,0.,1.}
\definecolor{ffqqqq}{rgb}{1.,0.,0.}
\scalebox{0.6}{
\begin{tikzpicture}[line cap=round,line join=round,>=triangle 45,x=1.0cm,y=1.0cm]
\draw[->,color=black] (-1.,0.) -- (12.,0.);
\foreach \x in {-1.,1.,2.,3.,4.,5.,6.,7.,8.,9.,10.,11.}
\draw[shift={(\x,0)},color=black] (0pt,2pt) -- (0pt,-2pt);
\draw[->,color=black] (0.,-6.) -- (0.,6.);
\foreach \y in {-6.,-5.,-4.,-3.,-2.,-1.,1.,2.,3.,4.,5.}
\draw[shift={(0,\y)},color=black] (2pt,0pt) -- (-2pt,0pt);
\clip(-1.,-6.) rectangle (12.,6.);
\draw[color=ffzztt,fill=ffzztt,fill opacity=0.25] {[smooth,samples=50,domain=1.5:8.0] plot(\x,{0-sqrt(16.0-(\x-4.0)^(2.0))})} -- (8.,0.) {[smooth,samples=50,domain=8.0:1.5] -- plot(\x,{sqrt(16.0-(\x-4.0)^(2.0))})} -- (1.5,-3.122498999199199) -- cycle;
\draw [dash pattern=on 5pt off 5pt] (4.,0.) circle (4.cm);
\draw [dash pattern=on 5pt off 5pt] (0.5,-6.) -- (0.5,6.);
\draw [dash pattern=on 5pt off 5pt] (5.5,0.) circle (5.5cm);
\draw [dash pattern=on 5pt off 5pt] (1.5,-6.) -- (1.5,6.);
\begin{scriptsize}
\draw [fill=qqqqff] (4.,0.) circle (2.5pt);
\draw[color=qqqqff] (4.48,0.42) node {\large{$\gamma^{-1}=\gamma_-^{-1}$}};
\draw [fill=ffqqqq] (0.5,0.) circle (2.5pt);
\draw[color=ffqqqq] (0.98,0.42) node {\large{$\frac{\alpha_-}{2}$}};
\draw [fill=qqqqff] (1.5,0.) circle (2.5pt);
\draw[color=qqqqff] (1.98,0.42) node {\large{$\frac{\alpha}{2}=\frac{\alpha_+}{2}$}};
\draw [fill=ffqqqq] (5.5,0.) circle (2.5pt);
\draw[color=ffqqqq] (5.98,0.42) node {\large{$\gamma_+^{-1}$}};
\end{scriptsize}
\end{tikzpicture}
}
\caption{The region containing $\sigma(S)\setminus\dR$ given by the set in the
right-hand side of \eqref{minimal_spec}.}
\label{}
\end{figure}
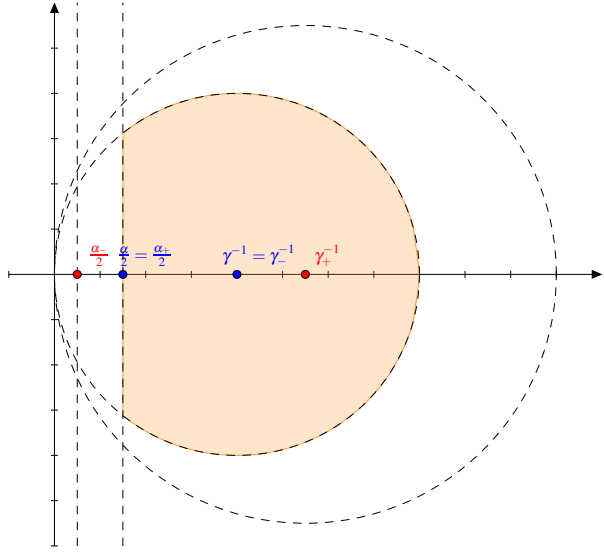

\section{Square root of a $J$-frame operator: applications to $J$-frames}
\label{sec:squareroot}

\noindent
It follows from Theorem~\ref{th:spec_incl} that if $S$ is a $J$-frame operator
acting on a Krein space $\cH$, then its spectrum is located in the open right half-plane.
Consider $f(z)=z^{1/2}$ as an analytic function on the open right half-plane
such that $\RE f(z)>0$ if $\RE z>0$.
Then, we can construct a square root of $S$ by means of the Riesz--Dunford
functional calculus:
\begin{equation}\label{sqrt}
  S^{1/2}\defeq\frac{1}{2\pi i}\int_\Gamma z^{1/2} (z-S)^{-1} \rd z,
\end{equation}
where $\Gamma$ is a Jordan curve in the right half-plane enclosing $\sigma(S)$.
The operator $S^{1/2}$ is also an invertible self-adjoint operator in the Krein space $\cH$.

The following theorem shows that this is (in some sense) the unique square root of $S$.

\begin{theorem}\label{unique sqrt}
Let $S$ be a $J$-frame operator in a Krein space $(\cH,\Skindef)$.
Then there exists a unique (bounded) operator $P$ acting on $\cH$ that
satisfies $P^2=S$ and
\begin{equation}\label{sp sqrt}
  \sigma(P) \subseteq
  \big\{z\in\dC: \ z=re^{it}, \ r>0,\ t\in(-\tfrac{\pi}{4},\tfrac{\pi}{4})\ \big\}.
\end{equation}
Moreover, such an operator is self-adjoint in the Krein space $(\cH,\Skindef)$ and
denoted by $S^{1/2}$.
\end{theorem}

\begin{proof}
By the discussion above and the Spectral Mapping Theorem, the operator $S^{1/2}$
defined in \eqref{sqrt} satisfies the desired conditions.
In particular, $S^{1/2}$ is invertible and we denote its inverse by $S^{-1/2}$.

Assume that $P$ is another operator satisfying $P^2=S$ and \eqref{sp sqrt}.
Then, $PS=SP$ and also
\[
  PS^{-1/2}=S^{-1/2}P,
\]
as $g(z)=z^{-1/2}$ is an analytic function in an open domain
containing $\sigma(S)$; see \cite[Proposition~4.9]{C85}.
Thus, the operator $U\defeq PS^{-1/2}$ satisfies
\[
  U^2=(S^{-1/2}P)(PS^{-1/2})=S^{-1/2}\,S\,S^{-1/2}=I.
\]
Hence $\sigma(U)\subseteq\{-1,1\}$.

On the other hand, both $\sigma(P)$ and $\sigma(S^{-1/2})$ are contained in the
open domain defined in the right-hand side of \eqref{sp sqrt} and,
by \cite[Lemma~0.11]{RR03},
\[
  \sigma(U) \subseteq \{\lambda\cdot\mu:\, \lambda\in \sigma(P), \,
  \mu\in \sigma(S^{-1/2})\}\subseteq \{z\in\dC: \ \RE(z)>0\}.
\]
So, $\sigma(U)=\{1\}$, or equivalently, $U=I$ and $P=S^{1/2}$.
\end{proof}

In the following we characterize the family of $J$-frames for $\cH$ with a
prescribed $J$-frame operator, i.e.\ given a $J$-frame operator $S:\cH\ra\cH$
we describe those $J$-frames $\cF$ with synthesis operator $T:\ell_2(I)\ra \cH$
such that $TT^+=S$.  To do so, we show that the synthesis operator of
a $J$-frame admits a polar decomposition (in the Krein space sense).

\medskip

First, let us state some well-known facts about partial isometries in Krein spaces.
Given Krein spaces $\cH$ and $\cK$, a (bounded) operator $U:\cH\ra\cK$ is
a \emph{partial isometry} if $UU^+U=U$, an \emph{isometry} if $U^+U = I$,
a \emph{co-isometry} if $UU^+=I$, and a \emph{unitary operator} if both $U^+U=I$ and $UU^+=I$.

For instance, $U$ is a partial isometry if and only if there exist
regular subspaces $\cM$ of $\cH$ and $\cN$ of $\cK$ such that
\begin{enumerate}
\item[(1)] $U$ is a one-to-one mapping from $\cM$ onto $\cN$ satisfying
  \begin{equation*}
    {\K{Uf}{Ug}}_{\cK} =  {\K{f}{g}}_{\cH} , \qquad f,g\in\cM;
   \end{equation*}
\item[(2)] $N(U) = \cM^{\ort}$.
\end{enumerate}
In this case, we call $\cM$ and $\cN$ the \emph{initial} and \emph{final spaces for~$U$}.
Moreover, $U^+$ is a partial isometry with initial space $\cN$ and final space $\cM$, and
\begin{equation*}
\begin{aligned}
  U^+U &= E_\cM \,, &N(U) &= N(E_\cM)=\cM^{\ort}, \\
  UU^+ &= E_\cN \,, &N(U^+) &= N(E_\cN)=\cN^{\ort},
\end{aligned}
\end{equation*}
where $E_\cM$ and $E_\cN$ are the self-adjoint projections onto $\cM$ and $\cN$,
respectively; see, e.g.\ \cite[Theorems~1.7 and 1.8]{DR}.

\medskip

Before stating the polar decomposition for the synthesis operator
of a $J$-frame $\cF$, recall that we are considering $\ell_2(I)$ as
a Krein space with the indefinite inner product induced by $\cF$; see \eqref{l2}.

\begin{proposition}\label{polar desc}
Let $\cF$ be a $J$-frame for a Krein space $(\cH,\Skindef)$ with
synthesis operator $T:\ell_2(I)\ra \cH$.  Then $T$ admits a polar decomposition
in the Krein space sense: there exists a co-isometry $U:\ell_2(I)\ra\cH$
with initial space $N(T)^{\ort}$ such that
\begin{equation*}
  T=(TT^+)^{1/2} U,
\end{equation*}
where the square root is the one defined in \eqref{sqrt}.
Moreover, this factorization is unique in the following sense:
if $T=PW$ where $P$ is a self-adjoint operator in $\cH$ satisfying \eqref{sp sqrt}
and $W:\ell_2(I)\ra\cH$ is a co-isometry, then $P=(TT^+)^{1/2}$ and $W=U$.
\end{proposition}

\begin{proof}
If $S=TT^+$ is the $J$-frame operator of $\cF$ and $S^{1/2}$ denotes
the self-adjoint square root defined in \eqref{sqrt},
then $U\defeq S^{-1/2}T$ is a co-isometry from the
Krein space $(\ell_2(I), \Skindef_2)$ onto $(\cH,\Skindef)$.
Indeed, it is immediate that
\[
  UU^+=(S^{-1/2}T)(T^+S^{-1/2})=S^{-1/2}\,S\, S^{-1/2}=I.
\]
Hence $U^+U$ is a self-adjoint projection in the Krein space $(\ell_2(I),\Skindef_2)$.
Also, $UU^+U=U$ implies that $N(U^+U)=N(U)=N(T)$, which is a
regular subspace of $\ell_2(I)$, see \cite[Lemma~4.1]{GMMM12}.
Hence, $N(T)^{\ort}$ is the initial space of $U$ and
\[
  T=S^{1/2}(S^{-1/2}\, T)=S^{1/2}\, U=(TT^+)^{1/2}\, U.
\]
Finally, let us show the uniqueness of such a factorization.
Assume that $T=PW$, where $P$ is a self-adjoint operator in $\cH$
satisfying \eqref{sp sqrt} and $W:\ell_2(I)\ra\cH$ is a co-isometry. Then,
\[
  TT^+= (PW)W^+P= P(WW^+)P=P^2.
\]
By Theorem \ref{unique sqrt}, we have that $P=S^{1/2}=(TT^+)^{1/2}$
and $U=S^{-1/2}T=S^{-1/2}PW=W$.
\end{proof}

If $\cF=\{f_i\}_{i\in I}$ is a $J$-frame for $\cH$ with
synthesis operator $T:\ell_2(I)\ra \cH$ and $J$-frame operator $S$,
then the synthesis operator of the canonical dual $J$-frame $\cF'=\{S^{-1}f_i\}_{i\in I}$
is given by $S^{-1}T$.  Therefore the following corollary is true.

\begin{corollary}
Let $\cF=\{f_i\}_{i\in I}$ be a $J$-frame for $\cH$ with
synthesis operator $T:\ell_2(I)\ra \cH$.
If $T$ is factorized as $T=(TT^+)^{1/2}U$ with a co-isometry $U:\ell_2(I)\ra \cH$,
then the canonical dual $J$-frame $\cF'=\{(TT^+)^{-1}f_i\}_{i\in I}$
has synthesis operator $(TT^+)^{-1/2}U:\ell_2(I)\ra \cH$.
\end{corollary}

Let a $J$-frame operator $S$ on a Krein space $(\cH,\Skindef)$ be given.
Then, by definition, there exists
a $J$-frame $\cF=\{f_i\}_{i\in I}$ with synthesis operator $T:\ell_2(I)\ra \cH$
such that $TT^+=S$.  In general, for a given $J$-frame operator $S$
there exist many $J$-frames such that $S$ is the associated $J$-frame operator.
The relation between two such $J$-frames with the same $J$-frame operator $S$
can be expressed via their corresponding synthesis operators, which is done
in the following theorem.

\begin{theorem}
Let $S$ be a $J$-frame operator in $(\cH,\Skindef)$.
Assume that $T_k:\ell_2(I_k)\ra \cH$ is the synthesis operator of a $J$-frame
such that $S=T_k T_k^+$, for $k=1,2$.
Then, there exists a partial isometry $W:\ell_2(I_2)\ra\ell_2(I_1)$
(in the Krein space sense) with initial space $N(T_2)^{\ort}$ and
final space $N(T_1)^{\ort}$ such that
\[
  T_2=T_1W.
\]
\end{theorem}

\begin{proof}
By Proposition \ref{polar desc}, there exist co-isometries $U_k:\ell_2(I_k)\ra \cH$
with initial spaces $N(T_k)^{\ort}$ such that $T_k=S^{1/2}U_k$.
Set $W\defeq U_1^+U_2$.  Then
\begin{align*}
  WW^+ &= (U_1^+U_2)(U_1^+U_2)^+=U_1^+(U_2U_2^+)U_1=U_1^+U_1=E_1, \\
  W^+W &= (U_1^+U_2)^+(U_1^+U_2)=U_2^+(U_1U_1^+)U_2=U_2^+U_2=E_2,
\end{align*}
where $E_k$ is the self-adjoint projection onto $N(T_k)^{\ort}$.
Therefore $W$ is a partial isometry (in the Krein space sense)
with initial space $N(T_2)^{\ort}$ and final space $N(T_1)^{\ort}$.
Furthermore, $T_1W=(S^{1/2}U_1)(U_1^+U_2)=S^{1/2}U_2=T_2$.
\end{proof}

Finally, we obtain in the last Corollary \ref{xmas}
of Proposition \ref{polar desc}  for a given $J$-frame operator $S$ a description
of all $J$-frames such that $S$ is the associated $J$-frame operator:
There is a bijection between all co-isometries and all $J$-frames
with the same $J$-frame operator.

\begin{corollary}\label{xmas}
Let $S$ be a $J$-frame operator in $(\cH,\Skindef)$.
Then, $T:\ell_2(I)\ra \cH$ is the synthesis operator of a $J$-frame
such that $S=T T^+$ if and only if there exists a co-isometry $U:\ell_2(I)\ra\cH$
such that
\[
  T=S^{1/2}U.
\]
\end{corollary}

\bigskip

\noindent
\textbf{Acknowledgements} \\
J.~Giribet, A.~Maestripieri, and F.~Mart\'{\i}nez Per\'{\i}a gratefully
acknowledge the support from CONICET PIP 0168.
J.~Giribet gratefully acknowledges the support from ANPCyT PICT-1365.
F.~Mart\'{\i}nez Per\'{\i}a gratefully acknowledges the support from UNLP 11X681.
L.~Leben gratefully acknowledges the support from the Carl-Zeiss-Stiftung.
Juan I. Giribet, Leslie Leben, Alejandra Maestripieri,
F.~Mart\'{\i}nez Per\'{\i}a, and C.~Trunk gratefully acknowledge the support
from Deutscher Akademischer Austauschdienst/Bundesministerium
f\"ur Bildung und Forschung, Project-ID 57130286.

\bibliographystyle{amsplain}

\end{document}